\documentclass[12pt,reqno]{amsart}

\usepackage[latin1]{inputenc}
\usepackage{amsmath,amsfonts,amssymb,graphics,enumerate,times}
\usepackage{amsthm,amscd}
\usepackage{latexsym,verbatim,graphicx,amsfonts}
\usepackage{hyperref}
\usepackage{multicol}
\usepackage[english]{babel}
\usepackage[T1]{fontenc}
\usepackage[margin=2cm]{geometry}

\theoremstyle{plain}
\theoremstyle{plain}
\newtheorem{teo}{Theorem}[section]
\newtheorem{coro}[teo]{Corollary}
\newtheorem{lem}[teo]{Lemma}
\newtheorem{pro}[teo]{Proposition}

\newtheorem{defn}[teo]{Definition}

\theoremstyle{definition}
\newtheorem{oss}[teo]{Remark}

\newcommand{\D}{\mathbb{D}}
\newcommand{\B}{\mathbb{B}}
\newcommand{\T}{\mathbb{T}}
\newcommand{\HH}{\mathbb{H}}
\newcommand{\N}{\mathbb{N}}
\newcommand{\C}{\mathbb{C}}
\newcommand{\rr}{\mathbb{R}}
\newcommand{\s}{\mathbb{S}}
\newcommand{\Z}{\mathbb{Z}}

\DeclareMathOperator{\Span}{span}

\newcommand{\p}{\partial}

\DeclareMathOperator{\RRe}{Re}

\DeclareMathOperator{\ext}{ext}

\newcommand{\RR}{\mathbb{R}}
\newcommand{\BB}{\mathbb{B}}
\newcommand{\CC}{\mathbb{C}}

\renewcommand{\SS}{\mathbb{S}}

\begin{document}
 
\title[Quaternionic inner and outer]{Quaternionic inner and outer functions}
\author[Monguzzi]{Alessandro Monguzzi}
\address{Dipartimento di Matematica, Universit\`a degli Studi di
  Milano, Via C. Saldini 50, 20133 Milano, Italy.}
 \email{alessandro.monguzzi@unimi.it}
\author[Sarfatti]{Giulia Sarfatti}
\address{Dipartimento di Matematica e Informatica``U. Dini'', Universit\`a di Firenze, viale {Morgagni} 67/A, 50134 Firenze, Italy.} \email{giulia.sarfatti@unifi.it}
\author[Seco]{Daniel Seco}
\address{Universidad Carlos III de Madrid and Instituto de Ciencias 
Matem\'aticas (CSIC-UAM-UC3M-UCM), Departamento de Matem\'aticas UC3M, Avenida de 
la Universidad 30,  28911 Legan\'es (Madrid),
Spain.} \email{dsf$\underline{\,\,\,}$cm@yahoo.es}
\date{}


\begin{abstract}
{We study properties of inner and outer functions in the Hardy space of the quaternionic unit ball. In particular, we give sufficient conditions as well as necessary ones for functions to be inner or outer.  }
\end{abstract}


\maketitle

\section{Introduction}\label{Intro}

An essential role in the function theory of the unit disk of the complex plane is played by the property that any function in the Hardy space factorizes (uniquely) as the product of an inner and an outer function. The connections of inner and outer functions are ubiquitous in mathematical analysis, ranging from operator theory to dynamical systems and PDEs (see \cite{CGP} and \cite{Gar}, for instance). One of the main reasons for this is the fact that the closed invariant subspaces (for the shift operator in the Hardy space) can be described via an identification with inner functions, whereas outer functions contain information about approximation properties, and, in fact, coincide with cyclic functions. In the recent paper \cite{ale-giulia}, the first two authors proved an inner-outer factorization theorem for the Hardy space of slice regular functions on the quaternionic unit ball $H^2(\mathbb B)$. Thus, it seems natural to investigate the properties of inner and outer functions in the quaternionic setting more deeply, and the present paper is a first step in this direction. We will see that some properties of holomorphic inner and outer functions are straightforwardly generalized to the quaternionic setting, whereas some other properties are more peculiar of slice regular functions.

The paper is organized as follows. In Section \ref{sec-notation} we fix the notation and we recall some basic definitions and properties of slice regular functions and the quaternionic Hardy space $H^2(\mathbb B)$. We devote Section \ref{sec-inner} to properties of inner functions, whereas in Section \ref{sec-outer} we focus on outer ones. Then, in Section \ref{sec-approx} we investigate cyclicity and properties of optimal approximant polynomials in the quaternionic setting. We conclude formulating some open problems in Section \ref{sec-problems}.

\section{Notation and basic definitions}\label{sec-notation}

In this section we recall a few definitions and properties of slice regular functions and the quaternionic Hardy space $H^2(\mathbb B)$. We do not include any proofs; we refer the reader to the monograph \cite{libroGSS} for the basics on slice regular functions and to \cite{deFGS} for results concerning $H^2(\mathbb B)$.

Let $\HH$ denote the skew field of quaternions, let $\BB=\{q\in\HH:\ |q|<1\}$ be the quaternionic unit ball and
let $\partial \BB$ be its boundary, containing elements of the form $q=e^{tI}=\cos t+\sin t I,\ I\in\SS,\ t\in\RR$, where $\SS=\{q\in \HH : q^2=-1\}$ is the two dimensional sphere of imaginary units in $\HH$. Then,
\[ \HH=\bigcup_{I\in \s}(\rr+\rr I),  \hskip 1 cm \rr=\bigcap_{I\in \s}(\rr+\rr I),\]
where the \emph{slice} $L_I:=\rr+\rr I$ can be identified with the complex plane $\C$ for any $I\in\s$.  
\medskip

A function $f:\BB\to\HH$ is a \emph{slice regular function} if the restriction $f_I$ of $f$ to $\BB_I:=\BB\cap L_I$
is {holomorphic}, i.e., it has continuous partial derivatives and it is such that
\[\overline{\p}_If_I(x+yI)=\frac{1}{2}\left(\frac{\p}{\p x}+I\frac{\p}{\p y}\right)f_I(x+yI)=0\]
for all $x+yI\in \BB_I$.
The relationship between slice regular functions and holomorphic functions of one complex variable is made clear in the following lemma.
\begin{lem}[Splitting Lemma]\label{splitting-lemma}
	If $f$ is a slice regular function on $\BB$, then, for every $I\in\SS$ and for every $J\in\SS$, $J$ orthogonal to $I$, there exist two holomorphic functions $F, G:\BB_I\to L_I$ such that for every $z=x+yI\in \BB_I$,
	$$
	f_I(z)= F(z)+G(z)J.
	$$
\end{lem}

It is a well-known fact that every slice regular function on the unit ball $\mathbb B$ admits a power series expansion of the form
$$
f(q)=\sum_{n\in\N} q^n a_n,
$$
where $\{a_n\}_{n\in\mathbb N}\subseteq\HH$. The \emph{conjugate} of $f$, which we denote by $f^c$, is the function defined by
\begin{equation}\label{regular-conjugate}
f^c(q):=\sum_{n\in\N}q^n\overline{a_n}.
\end{equation}
Morevover, we denote by $\widetilde{f}$ the function 
\begin{equation*}\label{function-tilde}
\widetilde f(q):=f(\overline q).
\end{equation*}
The function $\widetilde{f}$ is not slice regular but it is a {\em slice} function.
The class of slice functions was introduced in \cite{ghiloniperotti} in a more general setting than the present one.  In this paper we say that a function $f:\B\to \HH$ is a {\em slice function} if for any $I,J\in\SS$, 

\begin{equation}\label{repfor2}
\begin{aligned}
f(x+yI)&=\frac{1-IJ}{2}f(x+yJ)+\frac{1+IJ}{2}f(x-yJ).
\end{aligned}
\end{equation}
Slice regular functions are examples of slice functions. Moreover, Formula \eqref{repfor2} furnishes a tool to uniquely extend a holomorphic function defined on the complex disk $\B_J$ to a slice regular function defined on the whole unit ball $\B$ (see \cite{libroGSS}). Given $f_J:\B_J\to \HH$, holomorphic function with respect to the complex variable $x+yJ$, the function $\ext(f_J):\B\to \HH$ defined for any $x+yI\in \B$ as 
\[\ext(f_J)(x+yI)=\frac{1-IJ}{2}f_J(x+yJ)+\frac{1+IJ}{2}f_J(x-yJ)\]
is slice regular on $\B$.

\noindent Formula \eqref{repfor2} can also be used to prove the following result concerning the zeros of a slice regular function.
\begin{pro}\label{zerislice}
	Let $f$ be a slice regular function on $\B$ such that $f(\B_I)\subseteq L_I$ for some $I\in\s$. 
	If $f(x+yJ)=0$ for some $J\in\s\setminus \{\pm I\}$, then $f(x+yK)=0$ for any $K\in\s$.
\end{pro}

The structure of the zero set of a slice regular function is completely understood.
\begin{teo}
	Let $f$ be a slice regular function on  $\B$. If $f$ does not vanish identically, 
	then its zero set consists of the union of isolated points and isolated $2$-spheres of the form $x +y \mathbb{S}$ with 
	$x,y \in \mathbb{R}$, $y \neq 0$.
\end{teo} 

%
A $2$-dimensional sphere $x+y\s \subseteq \B$ of zeros of $f$ is called a 
{\em spherical zero} of $f$. 
Any point $x+yI$ of such a sphere is called a {\em generator} of the spherical zero $x+y\s$.
Any zero of $f$ that is not a generator of a spherical zero is called an
{\em isolated zero} 
of $f$.
Moreover, on each sphere $x+y\s$ contained in $\B$, the zeros of $f$ are in one-to-one correspondence with the zeros of $f^c$, see \cite[Proposition 3.9]{libroGSS}.
%
%

In general, the pointwise product of two slice regular functions is not a slice regular function and a suitable product must be considered, namely, the so-called \emph{slice} or $\ast$-\emph{product}. If $f(q)=\sum_{n\in\N} q^n a_n$ and $g(q)=\sum_{n\in\N} q^n b_n$ are two slice regular functions on $\B$, then
\begin{equation}\label{defn-product}
f\ast g(q):=\sum_{n\in\N}q^n\sum_{k\in\N}a_k b_{n-k}.
\end{equation}
This product is related to the pointwise product by the formula
\begin{equation}\label{star-product-pointwise}
f* g(q)=\left\{\begin{array}{l r}
0 & \text{ if $f(q)=0 $}\\
f(q)g(T_{f^c}(q)) & \text{if $f(q)\neq 0 $\,,}
\end{array}\right.
\end{equation}
where $T_{f^c}(q):= f(q)^{-1}q f(q)$. 
%

By means of the $\ast$-product, we can associate to a function $f$ its \emph{symmetrization} $f^s$, that is,
\begin{equation}\label{fs}
f^s(q):= f^c\ast f(q)=f\ast f^c(q).
\end{equation}
We remark here that the symmetrization $f^s$ is a \emph{slice preserving} function, namely $f^s(\mathbb B_I)\subseteq L_I$ for all $I\in\mathbb S$. In particular, this is equivalent to the fact that the coefficients in the power series expansion of $f^s$ are all real numbers.

Finally, we denote by $f^{-\ast}$ the inverse of $f$ with respect to the $\ast$-product, which is given by
\begin{equation*}
f^{-*}(q)=(f^s(q))^{-1}f^c(q).
\end{equation*}

\smallskip

The function $f^{-*}$ is defined on $\{q\in \B \ | \ f^s(q) \neq 0\}$ and $f* f^{-*}=f^{-*}* f =1$. All of $f^{c}, f^s$ and $f^{-\ast}$ are slice regular functions if $f$ is slice regular.

The basic theory of Hardy spaces $H^p(\mathbb B)$ was established in \cite{deFGS}. Here we only recall some facts for the Hilbert case $p=2$ and for the extremal case  $p=\infty$ since it is enough for our purposes. Set $\ell^2 := \ell^2(\N,\HH)$. The Hardy space $H^2(\B)$ is the function space defined as
$$
H^2(\B):=\bigg\{f\textrm{ slice regular on $\mathbb B$}: \,f(q)= \sum_{n\in \N}q^n a_n \,,\{a_n\}_{n\in \N} \in \ell^2 \bigg\}.
$$
Each function $f\in H^2(\mathbb B)$ admits a boundary value function, defined in a canonical sense almost everywhere (with respect to a measure $\Sigma$ that will be described later on). We still denote this function by $f$. With this in mind, if $f\in H^2(\B)$, whenever we write $f(q)$ with $|q|=1$, we are implicitly evaluating the boundary value function associated to $f$.

The space $H^2(\BB)$ is a right quaternionic Hilbert space with respect to the inner product \begin{equation}\label{innerproduct}
\Big<\sum_{n\in\N} q^n a_n,\sum_{n\in\N}q^n b_n\Big>:= \sum_{n\in\N}\overline{b_n}a_n.
\end{equation}

This inner product on $H^2(\B)$ admits also an integral representation. Let us endow $\partial\BB$ with the measure
\begin{equation}\label{Sigma}
d\Sigma\left(e^{tI}\right)=d\sigma(I)dt,
\end{equation}
where $dt$ is the Lebesgue measure on $[0,2\pi)$ and $d\sigma$ is the standard area element of $\s$, normalized in such a way that $\sigma(\s)=\Sigma(\p\B)=1$. Then,
\begin{equation}\label{integral inner product}
\left<f,g\right>= \int_{\p\B} \overline{g(q)}f(q)\, d\Sigma(q).
\end{equation}
The measure $\Sigma$, and not the induced Lebesgue measure on $\p\B$, is naturally associated to the Hardy space $H^2(\mathbb B)$, as reasoned in \cite{deFGS, metrica}. An important feature of this inner product is that it can be actually computed by restricting it to any slice $L_I$. In more detail, given any $I\in\SS$ we set
\begin{equation}\label{inner-product-slice2}
\left<f,g\right>_I=\frac{1}{2\pi}\int_{0}^{2\pi}\overline{g(e^{\theta I})}f(e^{\theta I})d\theta,
\end{equation}
where $d\theta$ is the Lebesgue measure on $[0,2\pi)$. For any $I\in\SS$, it holds that
\begin{equation}\label{inner-product-slice}
\left<f,g\right>=\int_{\p\B} \overline{g(q)}f(q)\, d\Sigma(q)=\left<f,g\right>_I.
\end{equation}
We denote by $H^\infty(\B)$ the space of bounded slice regular functions on the unit ball. Notice that $H^\infty(\B) \subseteq H^2(\B)$.

Definitions of inner and outer functions in the quaternionic setting are similar to the classical ones for holomorphic functions and first appeared in \cite{deFGS}. 
\begin{defn}\label{defn-inner}
	A function $\varphi\in H^\infty(\B)$ is {\em inner} if $|\varphi(q)|\leq 1$ on $\B$ and $|\varphi(q)|=1$ $\Sigma$-almost everywhere on $\p\B$.
\end{defn}
\begin{defn}\label{defn-outer}
	A function $g\in H^2(\B)$ is {\em outer} if given any $f\in H^2(\B)$ such that $|g(q)|=|f(q)|$ for $\Sigma$-almost every $q\in\p\B$, then $|g(q)|\geq |f(q)|$ for any $q\in\B$.
\end{defn}

We point out that the definition of inner and outer functions were given in terms of the induced Lebesgue measure  $m$ on $\p\B$. It is not difficult to show that $\Sigma$ and $m$ are mutually absolutely continuous.

The following theorem was proved in \cite{ale-giulia} by the first two authors.

\begin{teo}[Inner-outer factorization]\label{thm-factorization}
	Let $f\in H^2(\B)$, $f\not\equiv 0$. Then, $f$ has a factorization $f=\varphi*g$ where $\varphi$ is inner and $g$ is outer. Moreover, this factorization is unique up to a unitary constant in the following sense: if $f=\varphi\ast g=\varphi_1\ast g_1$, then $\varphi_1=\varphi\ast\lambda$ and $g_1=\overline{\lambda}\ast g$ for some $\lambda\in\HH$ such that $|\lambda|=1$.
\end{teo}

The proof of this theorem makes use of the concept of cyclicity. 

\begin{defn}\label{defn-cyclic}
	A function $g$ is {\em cyclic} in $H^2(\BB)$ if  
	\begin{equation}\label{outer}
	[g]:=\overline{\Span\left\{q^n\ast g, n\geq 0\right\}}=H^2(\BB).
	\end{equation}
\end{defn}
We stress out that $[g]$ is the smallest closed invariant subspace of $H^2(\B)$ containing $g$. Thus, $g$ is cyclic if the smallest closed subspace containing $g$ is the space $H^2(\mathbb B)$ itself.
In \cite{ale-giulia} it is firstly proved that each function $f\in\ H^2(\BB)$ admits a factorization $f=\varphi\ast g$ with $\varphi$ inner and $g$ cyclic. Afterwards, being cyclic is proved equivalent to being outer in the sense of Definition \ref{defn-outer}.

We remark that we work with right quaternionic Hilbert spaces, therefore the left-hand side of \eqref{outer} denotes the closure in $H^2(\BB)$ of elements of the form 
\[\sum_{n=0}^m (q^n\ast g)a_n=\sum_{n=0}^m (g\ast q^n)a_n=g*p_m,\] 
where $p_m$ is a quaternionic polynomial with scalar coefficients $\{a_n\}_{n=0}^m\subseteq\mathbb H$.

\section{Inner functions}\label{sec-inner}
Let us now focus on inner functions. To start, we would like to better understand any connection between $f$ being an inner function in $H^2(\mathbb B)$  and the properties of $f_I$ (the restriction of $f$ to the slice $L_I=\mathbb R+\mathbb RI$), or of the splitting components of $f$ (see Lemma \ref{splitting-lemma}).
Some of the results we include in this section are implicit in \cite{ale-giulia}. Here we state them explicitly and we make some remarks. 

We first prove a characterization of inner functions in $H^2(\mathbb B)$. In the following statement the $*$-product is the extension of \eqref{defn-product} to the more general setting of slice $L^2$ functions on $\p\B$, that is, the space $L_s^2(\p\B)$ of
functions of the form $q\mapsto \sum_{k\in \Z}q^ka_k$ with $q\in \p\B$ and $\{a_k\}_{k\in\Z}\in \ell^2$, see \cite{ale-giulia}. If $f(q)=\sum_{n\in\Z} q^n a_n$ and $g(q)=\sum_{n\in\Z} q^n b_n$ belong to $L_s^2(\p\B)$, then
\begin{equation}\label{product-slice}
f\ast g(q):=\sum_{n\in\Z}q^n\sum_{k\in\Z}a_k b_{n-k}. 
\end{equation}
Clearly, the boundary value function of every $f\in H^2(\B)$ is a slice $L^2$ function.
\begin{pro}\label{thmI}
	Let $f\in H^2(\B)$. Then, the following are equivalent:
	\begin{enumerate}[(i)]
		\item $f$ is inner;
		\item $\widetilde f\ast f^c=f^c\ast\widetilde f=1 \quad \Sigma$-almost everywhere on $\p\B$;
		\item there exists $I\in\s$ such that $(\widetilde f\ast f^c)_I=(f^c\ast\widetilde f)_I=1$ almost everywhere with respect to the induced Lebesgue measure on $\p\B_I$.
	\end{enumerate}  
\end{pro}
\vspace{-0.6cm}
\begin{proof}
	Recalling Lemma 2.4 in \cite{ale-giulia}, we know that $\widetilde f\ast f^c=f^c\ast\widetilde f=1$ $\Sigma$-almost everywhere on $\p\B$ if and only if $|f|=1$ $\Sigma$-almost everywhere on $\p\B$. Then, if $f$ is inner, we immediately get the first implication of the statement. 
	
	Suppose now that $\widetilde f\ast f^c=f^c\ast\widetilde f=1$ $\Sigma$-almost everywhere on $\p\B$, so that $|f|=1$ $\Sigma$-almost everywhere on $\partial \mathbb B$. Let $g\in H^2(\B)$, $g\not \equiv 0$, and denote by $Z_{g^s}=\{q\in\B \, : \, g^s(q)=0\}$. Proposition 2.3 in \cite{ale-giulia} guarantees that $\Sigma(Z_{g^s})=0$, whereas Proposition 5.32 in \cite{libroGSS} guarantees that the map $T_{g}:\p\B\setminus Z_{g^s}\to\p\B\setminus Z_{g^s}$ is a bijection. Therefore, we have
	\begin{equation*}
	\begin{aligned}
	\|f*g\|^2_{H^2}&\!=\!\|g^c*f^c\|^2_{H^2}\!=\!\int_{\p\B\setminus Z_{g^s}}\!\!\!\!|g^c(q)|^2|f^c(T_g(q))|^2d\Sigma(q)=\!\!\!\!\!\int_{\p\B\setminus Z_{g^s}}\!\!\!|g^c(q)|^2d\Sigma(q)=\|g^c\|^2_{H^2}=\|g\|^2_{H^2},
	\end{aligned}
	\end{equation*}
	where we used that  $|f|=1$ $\Sigma$-almost everywhere on $\p\B$ if and only if the same holds true for $|f^c|$ (see \cite[Proposition 5]{DRGS}). This implies that $f$ is a multiplier for $H^2(\B)$. Thanks to \cite[Corollary 3.5]{ABCS}
	we conclude that $f\in H^{\infty}(\B)$ and, in particular, that $f$ is an inner function. Hence, conditions $(i)$ and $(ii)$ are equivalent.
	
	Clearly $(ii)$ implies $(iii)$. Suppose now that condition $(iii)$ holds. Then, for almost every $t\in[0,\pi)$, we have both 
	\[1=\widetilde f\ast f^c(e^{tI}) \quad \text{and} \quad 1=\widetilde f\ast f^c(e^{(t+\pi)I})=\widetilde f\ast f^c(e^{-tI}).\]
	Using Formula \eqref{repfor2} we obtain that, for any $J\in \s$,
	\[\widetilde f\ast f^c(e^{tJ})=\frac{1-JI}{2}+\frac{1+JI}{2}=1.\]
	Recalling that $d\Sigma(e^{tI})=dtd\sigma(I)$, see \eqref{Sigma}, we immediately get $(ii)$. 
\end{proof}

\begin{oss}\label{remark}
	We point out that the previous proof actually showed that condition $(iii)$ implies $(\widetilde f\ast f^c)_I=(f^c\ast\widetilde f)_I=1$ almost everywhere on $\p\B_I$ for \emph{any} $I\in\s$.
\end{oss}
By means of the previous result we obtain another characterization of inner functions in $H^2(\B)$ which is often used as the definition of inner functions in more abstract Hilbert spaces (see, for instance, \cite{secochar}). Recalling the notations from \eqref{innerproduct} and \eqref{inner-product-slice2} and denoting by $\delta_k(j)$ the Kronecker delta, we have the following theorem.
\begin{teo}\label{thm-condition-inner}
	Let $f\in H^2(\B)$. Then, the following are equivalent:
	\begin{enumerate}[(i)]
		\item $f$ is inner;
		\item for all $k \in \N$, we have 
		\[ \big< q^k\ast f,f\big>= \delta_k(0);\] 
  		\item there exists $I\in\SS$ such that, for all $k \in \N$, we have 
		\[\big< q^k\ast f,f\big>_I= \delta_k(0).\]
	\end{enumerate}
\end{teo}
\vspace{-0.6cm}
\begin{proof}
	Let $f(q)=\sum_{n\in \N}q^na_n\in H^2(\B)$. Then, for $q\in\p\B$, we get 
	$$
	\big< q^k\ast f,f\big>=\Big< \sum_{n\ge 0}q^{n+k}a_n,\sum_{n\ge 0}q^na_n \Big> =  \sum_{n\geq \max\{0,k\}}\overline{a_n}a_{n-k}.
	$$
	We point out that $q^k\ast f$ has to be interpreted as a $\ast$-product in the setting of slice $L^2$ functions as defined in \eqref{product-slice}. Moreover, for $\Sigma$-almost any $q\in\partial\B$, it holds that
	\begin{equation*}
	f^c(q)=\sum_{n\geq 0} q^n\overline{a_n}\qquad\text{ and }\qquad \widetilde f(q)=\sum_{n\leq 0}q^n a_{-n}.
	\end{equation*}
	Hence,
	\begin{equation*}
	f^c\ast\widetilde f(q)=\sum_{k\in \Z}q^k\sum_{n\in\Z} \overline{a_n}a_{n-k}=\sum_{k\in \Z}q^k\sum_{n\geq \max\{0,k\}} \overline{a_n}a_{n-k}.
	\end{equation*}
	Therefore, for any $k\in\N$, we obtain that $\big< q^k\ast f,f\big>$ is the $k$-th coefficient in the power series expansion of $f^c\ast\widetilde f$. From this fact and Proposition \ref{thmI} is now easy to deduce that $(ii)$ follows from $(i)$. The reverse implication can be proved using the natural extension of the $H^2$ inner product to the bigger space of slice $L^2$ functions, namely, 
	\[\Big<\sum_{n\in \Z}q^na_n, \sum_{n\in\Z}q^nb_n \Big>_{L^2}=\sum_{n\in\Z}\overline{b_n}a_n.\] 
	In fact, suppose $(ii)$ holds. Then, the $k$-th coefficient in the power series expansion of $f^c\ast\widetilde f$ vanishes for $k>0$ and equals $1$ for $k=0$. To show that all the coefficients with $k=-n<0$ equal zero consider  
	\[\langle q^{-n}\ast f,f \rangle_{L^2}=\langle f,q^{n}\ast f \rangle_{L^2}=\overline{\langle q^{n}\ast f,f \rangle}=0,\]
	thus the first part of the theorem is proved.  
	
	From \eqref{inner-product-slice}, we obtain that $(ii)$ is equivalent to $(iii)$. In fact, the inner product of $H^2(\B)$ can be computed on a single slice and it does not depend on the choice of the slice.
\end{proof}

In general, the restriction $f_I$ of $f$ to the slice $L_I$ is a function of one complex variable, but still quaternion-valued. If $f_I$ were a complex-valued function, then $f_I$ would truly be an inner function of $H^2(\mathbb D)$. Therefore, Theorem \ref{thm-condition-inner} guarantees that the restriction to any slice $L_I$ of a slice regular inner function of $H^2(\B)$ is ``almost'' an inner function of $H^2(\mathbb D)$.  At this point it is natural to question about a simple converse. We wonder whether any inner function $F\in H^2(\mathbb D)$ admits a slice regular extension $f:=\ext(F)$ to the unit ball $\B$ such that $f$ is inner for $H^2(\mathbb B)$. 
Here we are identifying $\D$ with $\B_i$.
We obtain an answer in the form of the following corollary.
\begin{coro}
	Let $F\in H^2(\mathbb D)$ be an inner function. Then, the slice regular extension $f=\ext(F)$ is an inner function of $H^2(\mathbb B)$. 
\end{coro}
\begin{proof}
	Since the inner product of $H^2(\B)$ can be computed on any slice, it is clear that $f=\ext(F)\in H^2(\mathbb B)$ whenever $F\in H^2(\mathbb D)$. The conclusion now follows from Theorem \ref{thm-condition-inner} since condition $(iii)$ is satisfied for $I=i$.
\end{proof}

We explicitly point out that not all the inner functions of $H^2(\mathbb B)$ trivially arise as the regular extension of some inner function of $H^2(\mathbb D)$. 
In fact, if a slice regular function $f$ is the extension of a complex inner function $F$, then it necessarily preserves the slice $L_i$, i.e. $f(\B_i)\subseteq L_i$. Recalling Proposition  \ref{zerislice}, we have that if a function preserves a slice, then all its isolated, non-spherical, zeros are contained in that slice. It is enough to take a slice regular Blaschke product which has at least two zeros that are not on the same slice (and neither on the same sphere). See \cite{deFGS} for an explicit construction of such a function.

So far we investigated how $f$ being an inner function in $H^2(\B)$ affects the restriction of $f$ to any slice. We also want to understand how being an inner function affects the splitting components of the function (see Lemma \ref{splitting-lemma}). The following is our best result in this direction.

\begin{teo}\label{thm-inner-components}
	Let $f\in H^2(\B)$, $I,J\in\SS$, $J$ orthogonal to $I$, and $F,G:\B_I\to\CC$ be holomorphic functions so that, for any $z\in\B_I$,
	\begin{equation}\label{thm-splitting}
	f_I(z)=F(z)+G(z)J.
	\end{equation}
	Then, $f$ is inner if and only if for Lebesgue-almost every $x\in\partial \B_I$, the following conditions hold:\begin{equation}\label{eqn13}
	\left\{\begin{array}{l r}
	|F(z)|^2+|G(z)|^2=1  \\
	
	\phantom{\footnotesize{sss}}\\
	
	F(z)G(\overline{z})=F(\overline{z})G(z).
	\end{array}\right. 
	\end{equation}
\end{teo}
\begin{proof}
	
	If $F,G$ are the splitting components of $f_I$ as in \eqref{thm-splitting}, then, following \cite[Chapter 1]{libroGSS}, we get
	\begin{equation}\label{conditions-system}
	f^c_I(z)= \overline{F(\overline{z})}-G(z)J\qquad\text{ and }\qquad  \widetilde{f_I}(z)= F(\overline{z})+G(\overline z)J.
	\end{equation}
	Consider now the power series expansion  
	\[f(q)=\sum_{n\in \N}q^na_n=\sum_{n\in \N}q^n(\alpha_n+\beta_nJ),\]
	where $\alpha_n,\beta_n\in L_I$. Then,
	\[F(z)=\sum_{n\in \N}z^n\alpha_n, \quad \overline {F(\overline z)}=\sum_{n\in \N}z^n\overline{\alpha_n},\quad G(z)=\sum_{n\in \N}z^n\beta_n, \quad \overline{G(\overline{z})}=\sum_{n\in \N}z^n \overline{\beta_n}.\]
	Hence, for almost every $z\in \p\B_I$,
	\begin{equation}\label{splitting-product}
	\begin{aligned}
	(\widetilde f\ast f^c)_I(z)&= \Big(\sum_{n\le 0}z^{n}(\alpha_{-n}+\beta_{-n}J)\Big)*\Big(\sum_{n\ge 0}z^n(\overline{\alpha_n}-\beta_nJ)\Big)\\
	&= \sum_{n\in \Z}z^{n}\sum_{k\leq \min\{0,n\}}(\alpha_{-k}+\beta_{-k}J)(\overline{\alpha_{n-k}}-\beta_{n-k}J)\\
	&=\sum_{n\in \Z}z^{n}\sum_{k\leq \min\{0,n\}}(\alpha_{-k}\overline{\alpha_{n-k}}+\beta_{-k}\overline{\beta_{n-k}})+(\beta_{-k}\alpha_{n-k}-\alpha_{-k}\beta_{n-k})J\\
	&=F(\overline z)\overline{F(\overline z)}+ G(\overline z)\overline{G(\overline z)}+(G(\overline z)F(z)-F(\overline z)G(z))J.
	\end{aligned}
	\end{equation}
	Combining \eqref{splitting-product} with Remark \ref{remark} we get 
	$$
	F(\overline z)\overline{F(\overline z)}+ G(\overline z)\overline{G(\overline z)}+(G(\overline z)F(z)-F(\overline z)G(z))J=1
	$$
	for almost every $z\in\p\B_I$. This holds if and only if \eqref{eqn13} is satisfied.
\end{proof}

We conclude this section showing that the characterization of inner functions in \cite{secochar} involving their $H^2$ and $H^\infty$ norms works in the quaternionic setting as well.
\begin{teo}
	Let $f\in H^2(\B)$. The following are equivalent:
	\begin{enumerate}[(i)]
		\item $f$ is inner;
		\item $\|f\|_{H^2}=\|f\|_{H^\infty}=1$;
		\item $\|f\|_{H^2}=1$ and for all $k \in \N$ and $\lambda \in \HH$ we have
		\[ \| (q^k+\lambda) \ast f\|_{H^2} \leq \|q^k + \lambda\|_{H^2}.\]
	\end{enumerate}
\end{teo}
\vspace{-0.6cm}
\begin{proof}
	Clearly, if $f$ is inner, then its $H^\infty$ norm equals $1$ and
	\[\|f\|^2_{H^2}=\int_{\p\B}|f|^2d\Sigma=1,\]
	that is, $(i)$ implies $(ii)$. We deduce that $(ii)$ implies $(iii)$ from the fact that the multiplier space of $H^2(\mathbb B)$ can be isometrically identified with $H^\infty(\mathbb B)$ and the fact that $\|f\|_{H^\infty}=1$ implies $\|g\ast f\|_{H^2}\leq \|g\|_{H^2}$ for any $g\in H^2(\mathbb B)$.
	To see that $(iii)$ implies $(i)$ we exploit Theorem \ref{thm-condition-inner}. If $\|f\|_{H^2}=1$ and $f$ is not inner, then there must be $k \in \N \backslash \{0\}$ such that \[\big< q^k *f , f \big> \neq 0.\] Choose such $k$ and notice that, for all $\lambda \in \HH$, we have \[\|q^k + \lambda\|_{H^2}^2 = 1 + |\lambda|^2.\]  Let us now compute the left-hand side in the condition in $(iii)$ with a $\lambda$ to be chosen later. It holds that
	\[ \| (q^k+\lambda) \ast f\|_{H^2}^2= \|q^k \ast f\|_{H^2}^2 + |\lambda|^2 \|f\|_{H^2}^2 + 2 \RRe (\overline{\lambda} \big< q^k *f , f \big>).\]
	Since the shift is an isometry over its image on $H^2(\B)$, we see that the first two addends of the right-hand side sum to $\|q^k + \lambda\|_{H^2}^2$. Choosing $\lambda=\big< q^k *f , f \big>$ we contradict the hypothesis in $(iii)$.
\end{proof}

\section{Outer functions}\label{sec-outer}

In \cite{ale-giulia} it was shown that the Definition \ref{defn-outer} of outer functions is equivalent to the concept of cyclicity, extending a celebrated result of Beurling. Recall that a function $g \in H^2(\B)$ is cyclic if $[g]$, the smallest (closed) subspace of $H^2(\B)$ invariant under the action of the shift, is all of $H^2(\B)$.  What is currently missing in the quaternionic setting is an analogous of the classical characterization of outer functions on the unit disk in terms of the logarithm. Namely, $f\in H^2(\mathbb D)$ is outer if and only if 
\begin{equation}\label{outer-log}
f(z)= \alpha\exp\bigg\{\frac{1}{2\pi}\int_0^{2\pi}\log|f(e^{i\theta})|\frac{e^{i\theta}+z}{e^{i\theta}-z}\, d\theta\bigg\}
\end{equation}
for any $z\in\mathbb D$.
In this section we prove some preliminary results that go in the direction of finding an analogous characterization for quaternionic outer functions. We provide some necessary conditions as well as sufficient ones for a function to be outer. Some of these conditions are in terms of the symmetrization $f^s$ of the function $f$. We will see that, in some cases, $f$ being outer is equivalent to $f^s$ being outer. Since $f^s$ is slice preserving, a logarithm characterization for $f^s$ to be outer is available.
%

Most of our proofs rely on cyclicity, thus similar proofs could work for function spaces in which outer and cyclic functions do not necessarily coincide.

Our first finding does not come as a surprise, but it will be used in what follows.
\begin{lem}\label{conjucyc}
	Let $f \in H^2(\B)$. Then, $f$ is outer if and only if $f^c$ is outer.
\end{lem}

\begin{proof}
	Let $f= f_i\ast f_o$ be the inner-outer factorization of $f$ where $f_i$ denotes the inner part of $f$ and $f_o$ denotes the outer part. Assume that $f$ is outer, that is, $f=f_o$, so that $f^c=(f_o)^c$. A priori, we do not know if the conjugate of an outer factor is still an outer factor. Thus, assume for the moment that $f^c$ is not outer, that is, $f^c= (f^c)_i\ast (f^c)_o$.

	Then, on the one hand $f=f_o$, on the other hand
	$$
	f=(f^c)^c= ((f^c)_o)^c\ast ((f^c)_i)^c.
	$$
	In particular, thanks to \cite[Proposition 2.3]{ale-giulia}, for $\Sigma$-almost every $q\in\partial \mathbb B$, 
	$$
	|f(q)|=|f_o(q)|=|((f^c)_o)^c\ast ((f^c)_i)^c(q)|=|((f^c)_o)^c(q)|.
	$$
	Since $f$ is outer, we get
	$$
	|f(q)|=|f_o(q)|\geq |((f^c)_o)^c(q)|
	$$
	for any $q\in \mathbb B$. However, recall that a function is inner if and only if its conjugate function is inner (\cite[Proposition 2.1]{ale-giulia}), hence $|((f^c)_i)^c(q)|\leq 1$ inside the ball and we also get $$
	|f(q)|=|f_o(q)|=|((f^c)_o)^c(q)||((f^c)_i)^c(\widetilde T_{f^c}(q))|\leq |((f^c)_o)^c(q)|,
	$$
	where $\widetilde T_{f^c}(q)= ((f^c)_o)^c(q)^{-1}q ((f^c)_o)^c(q) $. We remark here that $(f^c)_o$ is never zero in $\mathbb B$ since it is the outer factor of $f^c$, hence $((f^c)_o)^c$ never vanishes as well. As a consequence,   $\widetilde T_{f^c}$ is a bijection of $\mathbb B$ to itself; see \cite[Proposition 5.32]{libroGSS}.
	
	Therefore, we obtain that $|((f^c)_i)^c|=1$ in $\mathbb B$, hence, for the maximum modulus principle in the quaternionic setting, we conclude that $((f^c)_i)^c\equiv\alpha$ where $\alpha$ is a quaternion of modulus $1$. Thus,
	$$
	f^c= \overline \alpha\ast (f^c)_o,
	$$
	that is, $f^c$ is outer and the proof is concluded.
\end{proof}

Let us now introduce the concept of optimal approximants which will be needed in what follows.
\begin{defn}
	Let $n\in\mathbb N$, $f\in H^2(\B)$ and $\mathcal{P}_n:=\{p(q)=\sum_{k=0}^nq^ka_k \ : \ a_k\in \HH \}$. A polynomial $p_n\in\mathcal P_n$ is an {\em optimal approximant} of degree $n$ of $f^{-*}$ if $p_n$ is such that $\|f*p_n-1\|_{H^2}=\min\{\|f*p-1\|_{H^2} \ : \ p \in \mathcal P_n\}$.
\end{defn}
\noindent The existence and uniqueness of such a minimizer is guaranteed by the projection theorem for quaternionic Hilbert spaces, see \cite{Jam}. In particular, the minimizer $f*p_n$ is given by the orthogonal projection of the constant function $1$ on the closed subspace  $f*\mathcal P_n \subsetneq H^2(\B)$. 

The constant function $1$ plays a special role because it is cyclic. Then, to show the cyclicity of a function $f$ is equivalent to show that its optimal approximants satisfy \begin{equation}\label{1-cyc}
\|f \ast p_n -1\|_{H^2} \rightarrow 0 \quad \text{as} \quad n \rightarrow \infty.
\end{equation}
In fact, if the constant function $1$ satisfies equation \eqref{1-cyc}, then it belongs to $[f]$. The fact that this is a closed and invariant subspace guarantees that $H^2(\B)=[1]\subseteq [f]$, that is, $f$ is cyclic.

The following result states the relationship between the invariant subspace generated by an $H^2$ function $f$ and the inner-outer factorization of $f$.
\begin{lem}\label{leminvariant}
	Let $f \in H^2(\B)$ factorizes as $f=f_i \ast f_o$, where $f_i$ is inner and $f_o$ is outer. Then, $[f]=[f_i]$.
\end{lem}

\begin{proof}
	Let $\{p_n\}_{n\in \N}$ be a sequence of polynomials such that $\|f_o \ast p_n -1\|_{H^2} \rightarrow 0$ as $n$ tends to $\infty$.
	Since $f_i$ is bounded, it is a multiplier, so that
	\[ \|f\ast p_n - f_i\|_{H^2}  \leq \|f_i\|_{H^\infty} \|f_o\ast p_n -1\|_{H^2}.\]
	This shows that $f_i \in [f]$ and hence $[f_i] \subseteq [f]$. The inclusion $[f] \subseteq [f_i]$ follows from the fact that $[f_i]=f_i\ast H^2(\mathbb B)$, and this latter space clearly contains $f$ since $f_o\in H^2(\B)$.
%
%
%
\end{proof}

\begin{lem}\label{prodcyc}
	Let $f, g \in H^\infty(\B)$. Then, $f \ast g$ is cyclic if and only if both $f$ and $g$ are cyclic.
\end{lem}
Before proving the lemma, notice that even in the complex case  the assumption $f, g \in H^\infty(\D)$ cannot be discarded since there are functions in $H^2(\D)$ whose square is not an element of $H^2(\D)$.

\begin{proof}
	To show that $f \ast g$ cyclic implies $g$ cyclic, thanks to Lemma \ref{conjucyc}, it is enough to show that $f \ast g$ cyclic implies $f$ cyclic and then apply the result to $g^c*f^c$. Suppose now that $f*g$ is cyclic and consider the inner-outer factorization of $f=f_i \ast f_o$ with $f_i$ inner and $f_o$ outer. Then, $[f]=f_i*H^2(\B)$. Hence, $f \ast g= f_i *(f_o*g)$ is an element of $[f]$, since $f_o\in H^2(\B)$ and $g\in H^\infty(\B)$ guarantee that $f_o*g\in H^2(\B)$  (see \cite{deFGS}).
	Then, $[f \ast g] \subseteq [f]$ but $[f \ast g] = H^2(\B)$. Hence, $f$ is cyclic. 
	
	Suppose now that $f$ and $g$ are cyclic, and let $\{p_n\}_{n \in \N}$ and $\{r_m\}_{m \in \N}$ be sequences of polynomials with the property that both the norms $\|f\ast p_n-1\|_{H^2}$ and $\|g\ast r_m-1\|_{H^2}$ tend to zero as $n$ goes to infinity. Then, for each $n, m \in \N$, from the triangle inequality we get
	\[\| (f \ast g) \ast (r_m \ast p_n) -1\|_{H^2} \leq \|(f \ast g) \ast (r_m \ast p_n) - f \ast p_n\|_{H^2} + \|f \ast p_n -1\|_{H^2}.\]
	
	The last term on the right-hand side will be arbitrarily small whenever $n$ is large enough. The other one may be estimated using the fact that $f$ and $p_n$ are both multipliers, that is,
	\[ \|(f \ast g) \ast (r_m \ast p_n) - f \ast p_n\|_{H^2} \leq  \|f\|_{H^\infty} \cdot \| g \ast r_m - 1\|_{H^2} \cdot \|p_n\|_{H^\infty}.\]
	Now, whatever the value of $\|f\|_{H^\infty}\|p_n\|_{H^\infty}$ is, it does not depend on $m$. Hence, if we fix $\varepsilon >0$ and $n \in \N$, taking $m$ large enough we obtain \[ \|  g \ast r_m - 1\|_{H^2} \leq \varepsilon (\|f\|_{H^\infty}\|p_n\|_{H^\infty})^{-1}.\] Therefore, $f \ast g$ is cyclic.
\end{proof}

The previous result will prove particularly useful when applied to the symmetrization $f^s=f\ast f^c= f^c\ast f$ of a function $f\in H^2(\mathbb B)$. 

For each $p \in [1,\infty]$, the function $f^s$ is in $H^p(\B)$ provided that $f$ is in $H^{2p}(\mathbb B)$. In particular, if $f \in H^\infty (\mathbb B)$, then $f^s \in H^\infty (\mathbb B)$, see \cite{deFGS}.
\begin{coro}\label{symmecyc}
	Let $f \in H^\infty(\mathbb B)$. Then $f$ is cyclic if and only if $f^s$ is cyclic.
\end{coro}
\begin{proof}
	If $f$ is bounded, so is $f^c$ (see \cite{DRGS}) and we can apply both Lemma \ref{prodcyc} and Lemma \ref{conjucyc}.
\end{proof}

As we mentioned, the importance of $f^s$ comes from the fact that it preserves slices, that is, it can be seen as a holomorphic complex-valued function on each slice. This is important because we can transfer the theory from the disk to the quaternionic ball.  In the following theorem we denote by $H^2(\mathbb B_I)$, $I\in\SS$, the function space defined as
$$
H^2(\mathbb B_I)=\left\{f\in \mathbb B_I\to L_I: f(z)=\sum_{n=0}^{\infty} z^n \alpha_n, \{\alpha_n\}\subseteq \ell^2(\mathbb N, L_I)\right\}.
$$
It is clear that $H^2(\mathbb B_I)$ can be identified with $H^2(\mathbb D)$.

\begin{teo}\label{thm1out}
	Let $f \in H^\infty(\mathbb B)$. The following are equivalent:
	\begin{enumerate}[(i)]
		\item $f$ is cyclic in $H^2(\mathbb B)$;
		\item $f$ is outer in $H^2(\mathbb B)$;
		\item $f^s$ is cyclic in $H^2(\mathbb B)$;
		\item $f^s$ is outer in $H^2(\mathbb B)$;
		\item $f^s_I$ is cyclic in $H^2(\mathbb B_I)$ for all $I \in \SS$;
		\item there exists $I\in \SS$ such that $f^s_I$ is cyclic $H^2(\mathbb B_I)$;
		\item $f^s_I$ is outer in $H^2(\mathbb B_I)$ for all $I \in \SS$;
		\item there exists $I\in \SS$ such that $f^s_I$ is outer in $H^2(\mathbb B_I)$.
	\end{enumerate}
\end{teo}
\vspace{-0.6cm}
\begin{proof}
	The equivalence between $(i)$, $(ii)$, $(iii)$ and $(iv)$ is guaranteed by Corollary \ref{symmecyc},  and by \cite[Theorem 4.2]{ale-giulia}. Also, since $f^s_I$ is a holomorphic function, the equivalence of $(v)$ and $(vii)$ and the equivalence of $(vi)$ and $(viii)$ are well-known consequences of the Beurling Theorem. \\
	Let us now prove that $(iii)$ implies $(v)$. Let $f^s$ be cyclic in $H^2(\B)$ and let $I\in\s$. Then, for any $g_I\in H^2(\B_I)$ there exists a sequence of quaternionic polynomials $\{p_n\}_{n\in\N}$ such that $\|f^s*p_n-\ext g_I\|_{H^2}$ tends to zero as $n$ goes to infinity. 
	Let now $p_n(z)=P_n(z)+Q_n(z)J$ be the splitting of $p_n$ with respect to $J\in\s$, $J$ orthogonal to $I$. Then, evaluating the $H^2$ norm on the slice $L_I$, we get 
	\begin{align*}
	\|f^s*p_n-\ext g_I\|^2_{H^2}&=\|f_I^s(P_n+Q_nJ)-g_I\|^2_{H^2(\B_I)}=\|f^s_IP_n-g_I\|^2_{H^2(\B_I)}+\|Q_n\|^2_{H^2(\B_I)},
	\end{align*}
	where the last equality is due to the orthogonality of $I$ and $J$. Therefore, the sequence of complex polynomials $\{P_n\}_{n\in\N}$ in the variable $z\in \B_I$  is such that 
	$\|f_I^sP_n-g_I\|_{H^2(\B_I)}$ tends to zero as $n$ goes to infinity, that is,  $f^s_I$ is cyclic in $H^2(\B_I)$.
	To conclude, it suffices to show that $(vi)$ implies $(iii)$, since clearly $(v)$ implies $(vi)$.  
	Suppose that $f^s_I$ is cyclic in $H^2(\mathbb B_I)$, for some $I\in\s$. Consider $g\in H^2(\B)$ and let $g(z)=F(z)+G(z)J$ be its splitting on $\B_I$ with respect to $J\in\s$, $J$ orthogonal to $I$. By hypothesis, there exist two sequences $\{P_n\}_{n\in\N}$ and $\{Q_n\}_{n\in\N}$ of complex polynomials in $\B_I$ such that $\|f^s_IP_n-F\|_{H^2(\B_I)}$ and $\|f^s_IQ_n-G\|_{H^2(\B_I)}$ tend to zero as $n$ goes to infinity. Then, using again the fact that the $H^2$ norm can be computed on any slice, and the orthogonality of $I$ and $J$, we get
	\begin{equation*}
	\begin{aligned}
	\|f^s*\ext(P_n+Q_nJ)-g\|^2_{H^2}&=\|f_I^s(P_n+Q_nJ)-g_I\|^2_{H^2(\B_I)}=\|f_I^s(P_n+Q_nJ)-(F+GJ)\|^2_{H^2(\B_I)}\\
	&=\|f_I^sP_n-F\|^2_{H^2(\B_I)}+\|f_I^sQ_n-G\|^2_{H^2(\B_I)}.
	\end{aligned}
	\end{equation*}
	This latter quantity tends to zero as $n$ goes to infinity, thus we can conclude that $f^s$ is cyclic in $H^2(\B)$.
\end{proof}

Notice that the notion of outer function in $(vii)$ and $(viii)$ is the classical one, hence it may be expressed as in \eqref{outer-log}, that is, in terms of the mean value property of the logarithm. We remark once again that the assumption $f\in H^\infty(\B)$ cannot simply be dropped, not only because we need $f^s$ to be defined as a $H^2(\B)$ function so that we can apply Beurling Theorem to it (and this would be guaranteed if $f\in H^4(\B)$, see \cite{deFGS}), but also because of the applicability of Lemma \ref{prodcyc}, for which we need $f \in H^\infty(\B)$.

The presence in Theorem \ref{thm1out} of the hypothesis $f \in H^\infty(\B)$ is likely unsatisfactory. At the moment, we have not been able to show any characterization in terms of mean value properties for more general $f$. However, sufficient conditions may be shown.
Given $\omega \in \B$ let $\tau_\omega$ denote the slice regular M\"obius transformation of the unit ball taking $0$ to $\omega$, see \cite{moebius},
\[\tau_\omega(q)=(1-q\overline{\omega})^{-*}*(\omega-q),\]
and let ${I_\omega}$ be the imaginary unit identified by $\omega$, that is, ${I_\omega}=\frac{\omega-\RRe \omega}{|\omega-\RRe \omega|}$ if $\omega$ is not real, ${I_\omega}$ is any imaginary unit otherwise. With this notation, it holds that
\[({\tau_\omega})_{{I_\omega}}(z)=(1-z\overline{\omega})^{-1}(\omega-z)\]
for any $z\in \B\cap L_{I_\omega}$.
\begin{pro}\label{thm2out}
	Let $f \in H^2(\B)$. Suppose that for all $\omega \in \B$ we have
	\begin{equation}\label{eqn401}
	\frac{1}{2 \pi} \int_{\p \B_{I_\omega}} \log |f_{I_\omega} \circ \tau_\omega(e^{\theta{I_\omega} })|d\theta = \log |f(\omega)|.
	\end{equation}
	Then, $f$ is outer.
\end{pro}

\begin{proof}
	Suppose that $g$ is a function such that on $\p \B$ we have $|f|=|g|$ $\Sigma$-almost everywhere, and let $\omega \in \B$. Then,
	\[ \log |f(\omega)| = \frac{1}{2 \pi} \int_{\p \B_{I_\omega}} \log |f_{I_\omega}\circ \tau_\omega(e^{\theta I_\omega})|d\theta. \]
	Since $|f|$ is equal to $|g|$ on the boundary, 
	the right-hand side is equal to
	\[\frac{1}{2 \pi} \int_{\p \B_{I_\omega}} \log |g_{I_\omega} \circ \tau_\omega(e^{\theta I_\omega})|d\theta.\]
	Notice that the composition $g_{I_\omega} \circ \tau_\omega$ is well defined on the slice $L_{I_\omega}$ and it is indeed the restriction of the slice regular function $\ext(g_{I_\omega} \circ \tau_\omega)$. Recalling that the logarithm of the modulus of a slice regular function is subharmonic (see \cite{deFGS}), we get
	\[\frac{1}{2 \pi} \int_{\p \B_{I_\omega}} \log |g_{I_\omega} \circ \tau_\omega(e^{\theta I_\omega})|d\theta \geq \log |g(\omega)|.\]
	All this together yields that $|f(\omega)| \geq |g(\omega)|$. Since $\omega$ was arbitrary, we conclude that $f$ is outer.
\end{proof}

\section{Optimal approximants}\label{sec-approx}

In this section we extend as much as possible the theory of optimal approximants to the quaternionic setting. A good account of the theory of such polynomials in the classical holomorphic setting is given in \cite{daniel-london}. 

Recall that, given $n\in\mathbb N$, $\mathcal{P}_n=\{p(q)=\sum_{k=0}^nq^ka_k \ : \ a_k\in \HH \}$.
The reproducing kernel of the subspace $f*\mathcal P_n$ exists since it is a closed subspace of $H^2(\B)$ which is itself a reproducing kernel Hilbert space with kernel function  $k(q,w)=(1-q\overline w)^{-*}$.
Let $\{f*\varphi_k\}_{k=0}^n$ be an orthonormal basis of $f*\mathcal P_n$, where $\varphi_k$ is a polynomial of degree $k$ for any $k=0,\ldots,n$. 
Then, from the reproducing property we can see that the reproducing kernel of $f*\mathcal P_n$ is given by
\[K_n(q,w)=\sum_{k=0}^{n}f*\varphi_k(q)\langle K_n(q,w), f*\varphi_k  \rangle=\sum_{k=0}^{n}f*\varphi_k(q)\overline{f*\varphi_k(w)}.\]
For more information about reproducing kernel Hilbert spaces in the quaternionic setting we refer the reader, for instance, to \cite{QRKHS}.

Notice that, since $f*p_n$ is the orthogonal projection of the constant function $1$, 
\[f*p_n(q)=\sum_{k=0}^{n}f*\varphi_k(q)\langle 1, f*\varphi_k  \rangle=\sum_{k=0}^{n}f*\varphi_k(q)\overline {f*\varphi_k(0) },\]
i.e.,
\begin{equation}\label{poly-kernel}
f*p_n(q)=K_n(q,0). 
\end{equation}
In particular, if $f(0)=0$, then $p_n \equiv 0$ for all $n\in \N$.
\begin{teo}
	Let $f\in H^2(\B)$ be such that $f(0)\neq 0$ and let $p_n$ be the optimal approximant of $f^{-*}$ of degree $n\in\mathbb N$. Then, all the zeros of $p_n$ lie outside the closed unit ball $\overline{\B}$.
\end{teo}
\begin{proof}
	First, let us show that we can reduce the problem to optimal approximants of degree $1$. Let $\lambda$ be a zero of an optimal approximant $p_n$ for the function $f \in H^2(\B)$. Then, there exists $\hat \lambda$ on the same two dimensional sphere of $\lambda$ such that $p^c_n(q)=(q-\overline{\hat\lambda})*\hat p_n^c$, so that $p_n(q)=\hat{p}_n*(q-\hat{\lambda})$, where $\hat{p}_n$ is a polynomial of degree $n-1$ and $|\hat\lambda|=|\lambda|$.  Then, the optimality of $p_n$ guarantees that $(q- \hat{\lambda})$ is the optimal approximant of degree 1 for the function $f*\hat{p}
	_n \in H^2(\B)$ which implies that $\hat{\lambda}$ is also a zero of a degree $1$ optimal approximant. 
	Therefore, in order to understand the possible positions of any such zero, it is enough to understand the same question for $n=1$. 
	Now, suppose that $p_1(q)=(q-\lambda)c$ is the optimal approximant of degree $1$ of $f^{-*}$. Then, by definition of orthogonal projection, $f*p_1 -1$ must be orthogonal to $f*q$, which translates easily in the equation \[0=\big< f*p_1, f*q \big> =\langle f*qc-f\lambda c, f*q \rangle,\] which implies that 
	\[
	\langle f*q, f*q \rangle c=\langle f, f*q \rangle \lambda c,
	\]
	that is,
	\begin{equation}\label{eqn301} |\lambda| = \frac{\|f*q\|^2}{|\left< f , f*q \right>|}. \end{equation}
	Notice that $f*q$ is never a multiple of $f$ unless $f \equiv 0$ (which is against our hypothesis), and hence we can apply Cauchy-Schwarz inequality as a strict inequality to $\langle f, f*q\rangle$ in \eqref{eqn301} to get
	\[ |\lambda| > \frac{\|f*q\|}{\|f\|}.\]
	Since $f*q=q*f$ and the shift is an isometry, the right-hand side is equal to $1$ and the proof is concluded.
\end{proof}

Notice that all the points outside the closed unit ball are zeros of some optimal approximants. Indeed, if $p_1(q) = q-\lambda$ with $|\lambda| > 1$, then $p_1^{-\ast} \in H^2(\B)$ and $\|p_1^{-\ast} \ast p_1 -1\|=0$. Therefore, $p_1$ must be the only optimal approximant.

We can further understand the relationship between optimal approximants and orthogonal polynomials in the spirit of \cite{daniel-london}.
\begin{teo}
	Let $f\in H^2(\B)$ and let $p_n$ be the optimal approximant of degree $n\in\mathbb N$ of $f^{-*}$. Let $\{f\ast\varphi_k\}_{k=0}^n$ be an orthonormal basis of $f\ast\mathcal P_n$, where $\varphi_k \in \mathcal P_k$. Then, the following are equivalent:
	\begin{enumerate}[(i)]
		\item$f$ is cyclic;
		\item $p_n(0)$ converges to $f^{-\ast}(0)$ as $n\to\infty$; 
		\item $\sum_{k=0}^{\infty}|\varphi_k(0)|^2=|f^{-\ast}(0)|^2$.
	\end{enumerate}
\end{teo}
\vspace{-0.6cm}
\begin{proof}
	Bearing in mind that $f*p_n$ is the orthogonal projection of the constant function $1$, we can see that \[\|f*p_n -1\|^2 = \big< 1- f*p_n, 1-f*p_n\big>=  \big< 1- f*p_n, 1\big>= 1-f(0)*p_n(0).\] 
	From this equality, the equivalence $(i)-(ii)$ is easily deduced.
	Also, from \eqref{poly-kernel} we see that either $(i)$ or $(ii)$ is equivalent to $1-K_n(0,0) \rightarrow 0$ as $n \rightarrow \infty$. However, $K_n(0,0)$ tending to $1$ is equivalent to $(iii)$ and this concludes  the proof.
\end{proof}

\section{Some open problems}\label{sec-problems}

We consider that the topic needs more development. We propose a few questions that seem natural from where we stand, beyond the obvious elimination of the boundedness hypothesis in Theorem \ref{thm1out}.
\begin{enumerate}
	\item[(A)] If $f\in H^2(\mathbb B)$, let $f= f_i\ast f_o$ and $f^c= (f^c)_i\ast (f^c)_o$ be the inner-outer factorizations of $f$ and of its conjugate function. Then we also have
	$f=f_i\ast f_o=[(f^c)_o]^c\ast [(f^c)_i]^c$. Is there any relationship between these two factorizations? 
	Is there something that can be said about the inner-outer factorization of $f^s$?
	\item[(B)] Suppose that the symmetrization $f^s$ of a function $f\in H^2(\B)$ is inner. Is it true that $f$ (or $f^c$) is inner?
	\item[(C)] Is the sufficient condition in Proposition \ref{thm2out} necessary for a function to be outer? This can be shown for $f^s$ under the assumption that $f$ is a multiplier.
	\item[(D)] The boundary values of slice components of a quaternionic inner function form what is usually called a \emph{Pythagorean pair}, a special situation in which two functions have modulus 1 everywhere when seen as one function in $\T^2$. Such pairs arise in connections with the so-called de Branges-Rovnyak spaces and other areas of mathematics. Can anything else be said about this relation at all?
\end{enumerate}

\bigskip

\small{\noindent\textbf{Acknowledgements.} 
	The first author is a member of INDAM-GNAMPA and is partially supported by the 2015 PRIN grant
	\emph{Real and Complex Manifolds: Geometry, Topology and Harmonic Analysis}  
	of the Italian Ministry of Education (MIUR). 
	
	The second author is partially supported by INDAM-GNSAGA, by the 2014 SIR grant {\em Analytic Aspects in Complex and Hypercomplex Geometry} and by Finanziamento Premiale FOE 2014 {\em Splines for accUrate NumeRics: adaptIve models for Simulation Environments} of the Italian Ministry of Education (MIUR).
	
	The third author is grateful for the financial support by
	the Severo Ochoa Programme for Centers of Excellence in R\&D
	(SEV-2015-0554) at ICMAT, and by the
	Spanish Ministry of Economy and Competitiveness, through grant
	MTM2016-77710-P.
	
	\smallskip
	
	Part of this project was carried out during a visit of the first and third author at the University of Firenze, and we wish to thank the Department of Mathematics and Computer Sciences for the financial support and the warm hospitality.
}

\end{document}